\documentclass{amsart}
\usepackage{amsmath,amsthm,amssymb}
\usepackage{mathrsfs}

\title[Characterizations of some Hartogs domains]{An application of a Diederich-Ohsawa theorem in characterizing some Hartogs domains}

\author{Hyeseon Kim and Atsushi Yamamori}

\address{The Center for Geometry and its Applications, Pohang University of Science and Technology, Pohang 790-784, Republic of Korea}

\email{hop222@snu.ac.kr, hyeseon@postech.ac.kr}

\address{The Center for Geometry and its Applications, Pohang University of Science and Technology, Pohang 790-784, Republic of Korea}

\email{yamamori@postech.ac.kr, ats.yamamori@gmail.com}



\thanks{The research of the authors was supported by an NRF grant $2011$-$0030044$~(SRC-GAIA) of the Ministry of Education, The Republic of Korea.}

\theoremstyle{definition}
\newtheorem{theorem}{Theorem}[section]
\newtheorem{lemma}[theorem]{Lemma}

\newtheorem{corollary}{Corollary}
\newtheorem{remark}{Remark}

\newtheorem{example}{Example}[section]

\newtheorem*{acknowledgement}{Acknowledgement}

\begin{document}
\begin{abstract}
Applying a theorem due to Diederich and Ohsawa on weighted Bergman kernels,
we characterize some Hartogs domains by their holomorphic automorphisms.
\end{abstract}

\maketitle

\section{Introduction}
Let $\Omega$ be a domain in $\mathbb{C}^n$ and denotes its automorphism group by $\mathrm{Aut}(\Omega)$, the set of all biholomorphic self maps that forms a group under the composition law. A celebrated theorem due to Riemann states that every simply connected domain $\Omega\subsetneq\mathbb C$ is holomorphically equivalent to the unit disc $\mathbb D=\{|z|<1\}$. It is natural to ask whether Riemann mapping theorem can hold in $\mathbb C^n$ for $n\geq2$. As a monumental result to this question, Poincar\'{e} showed the following in analyzing structures of automorphism groups: \emph{The unit ball and the polydisk in $\mathbb C^{n}$ cannot be holomorphically equivalent}. 

It indicates the difficulty of the holomorphic equivalence problem of higher dimensional cases. However, because of the biholomorphic invariance of the automorphism group, an accurate knowledge of $\mathrm{Aut}(\Omega)$ may enhance in understanding this problem.

In~\cite{D-O}, Diederich and Ohsawa considered a moment problem for weighted Bergman kernels which arises from classical issues in analysis; namely as an intermediate procedure of approaching the problem, they examined it to find out a suitable subclass of admissible weights $\varphi$ such that the corresponding weighted Bergman kernel $K_{\varphi}$ characterizes the weight $\varphi$ uniquely within this subclass. One such a subclass treated in~\cite{D-O}, concerns the complex polynomial ring $\mathbb C[z]$~(see Theorem~\ref{thm:D-O}). It is worthy of close attention that their result is (unexpectedly) connected to characterization problems of Hartogs domains.

The purpose of this paper is to characterize some Hartogs domains from the viewpoint of their automorphisms. The mappings in their automorphism groups about which we shall investigate, essentially interpret the weighted Bergman kernels of their base domains, as well. These associated weighted Bergman kernels invoke the moment problem by Diederich-Ohsawa, as a main ingredient of our approach in characterizing Hartogs domains.

This paper is organized as follows: In Section~\ref{pre} we first review the notion of the weighted Bergman kernel and then explain some basics of Hartogs domains in conjunction with a primary idea to demonstrate a methodological perspective in~\cite{D-O}. Besides, typical examples of weighted Bergman spaces are provided; precisely, the Fock-Bargmann space and weighted Bergman spaces on irreducible bounded symmetric domains. 

In Section~\ref{sec:Fock} we investigate the Fock-Bargmann-Hartogs domain, an unbounded strongly pseudoconvex domain which is non-hyperbolic in the sense of Kobayashi. Then, we establish Theorem~\ref{main1} as a main result, using Theorem~\ref{thm:D-O} and an explicit form~\cite{KTY} of its automorphism group. What is more, an alternative proof of Theorem~\ref{main1} is given in Remark~\ref{PL_condition} under a certain additional condition pertinent to a Phragm{\'e}n-Lindel{\"o}f type theorem for unbounded domains. 

As the second main result of this paper, we also give an analogue of Theorem~\ref{main1} to the Cartan-Hartogs domain~(see~Theorem~\ref{main2} in Section~\ref{CH_dom}).

In Section~\ref{C_Rmk} we close this paper by addressing, a remark on our argument.

\medskip

\section{Preliminaries}\label{pre}
Let $D$ be a domain in $\mathbb C^n$ and $p$ a positive measurable function on $D$.
The weighted Bergman space $A^2 (D, p)$ is the space of holomorphic functions in $L^2(D, p)$:
$$A^2 (D, p) := \left\lbrace f \in \mathcal O(D): \int_D |f(z)|^2  p(z)dV(z) < \infty \right\rbrace.$$
The weighted Bergman space $A^2 (D, p)$ is a Hilbert space with the following inner product:
$$ \langle f,g \rangle = \int_D f(z)\overline{g(z)} p(z)dV(z).$$
The function $p$ is called a weight of integration.
A weight of integration is called admissible if
\begin{itemize}
\item[(i)] $A^2(D, p)$ is a closed subspace of $L^2(D,p)$,
\item[(ii)] the evaluation functional $r_z: A^2(D, p) \rightarrow \mathbb C$ defined by $r_z(f)=f(z)$, is bounded linear.
\end{itemize}
These conditions ensure the existence of the reproducing kernel of $A^2(D, p) $.
Throughout this paper, we always assume that a weight of integration is admissible.
For further information of admissible weights, see \cite{PW}.
The reproducing kernel of $A^2(D, p) $ is called the weighted Bergman kernel of $D$ and denote it by $K_{D,p}$.
Namely, the kernel function $K_{D,p}$ is the unique function satisfying the reproducing property
\[
 f(z)= \int_D f(w) K(z,w) p(z)dV(w), \quad \mbox{for all } f \in A^2(D,p).
\]
Let $\{e_k\}_{k\geq 0}$ be a complete orthonormal basis of $A^2 (D, p)$. Then the weighted Bergman kernel is computed by a complete orthonormal basis:
\begin{align*}
K_{D,p}(z,w)=\sum_{k \geq 0} e_k (z) \overline{e_k (w)},
\end{align*}
which is independent of the choice of complete orthonormal basis. Especially, if $p \equiv 1$, then weighted Bergman kernel is called the (unweighted) Bergman kernel and denote it by $K_D$.
Let $F: D \rightarrow D'$ be a biholomorphism.
The following transformation law of the Bergman kernel is fundamental:
\begin{align}\label{trans}
K_D (z,w)= J(F,z)\overline{J(F,w)} K_{D'} (F(z),F(w)).
\end{align}
Here, $J(F,z)$ is the Jacobian determinant
$$ J(F,z) = \det \left(  \dfrac{\partial F_j}{\partial z_k} (z)\right)_{j,k=1,\ldots, n}.$$

Define the Hartogs domain $\Omega_p^D$ by
$$\Omega_p^D =\{ (z,\zeta)\in D \times \mathbb C^m: \|\zeta\|^2 < p(z) \}.$$
Then it is known that the Bergman kernel of $\Omega_p^D$ has a series expansion involving weighted Bergman kernels of its base domain $D$:
\begin{align}\label{FRC}
K_{\Omega_p^D} ( (z,\zeta), (z',\zeta')) = \dfrac{1}{\pi^m} \sum_{k \geq 0} (k+1)_m K_{D, p^{ k+m}} (z,z') \langle \zeta,\zeta' \rangle^k.
\end{align}
Specifically, the restriction of the Bergman kernel $K_{\Omega_p^D}$ to the space $\{ \zeta=\zeta'=0\}$ coincides with the weighted Bergman kernel $K_{D, p^{m}}$ (up to a constant multiple);
namely,
\begin{align}\label{restFRC}
K_{\Omega_p^D} ( (z,0), (z',0)) = \dfrac{m!}{\pi^m}  K_{D, p^{m}} (z,z').
\end{align}
Since the formula \eqref{FRC} was first proved by Forelli and Rudin \cite{F-R} for the case $D=\mathbb D$ and $p(z)=1-|z|^2$, it is called the Forelli-Rudin construction.
Note that Ligocka proved the Forelli-Rudin construction for general pairs $(D,p)$ in \cite{Ligocka}. 
We will see that the equation \eqref{restFRC} plays a substantial r\^{o}le in the proof of our main results.
\begin{remark}
Some generalizations of the Forelli-Rudin construction are known.
Engli\v{s} \cite{Englis2000} obtained a Forelli-Rudin type formula for the domain
$$\Omega_{\psi,\phi}=\left\lbrace (z, \zeta_1, \zeta_2)\in \Omega\times \mathbb C^{d_1} \times \mathbb C^{d_2}: \dfrac{\|\zeta_1\|^2 }{\psi(z)} + \dfrac{\|\zeta_2\|^2 }{\phi(z)}  <1 \right\rbrace. $$
Englis and Zhang \cite{Englis2006} gained a Forelli-Rudin type formula for the domain
$$\Omega_{\mathscr D}= \left\lbrace (z,\zeta) \in \Omega \times \mathbb C^d : \psi (z)^{-1/2} \zeta \in \mathscr D \right\rbrace, $$
where $\mathscr D$ is an irreducible bounded symmetric domain in $\mathbb C^d$.
Another generalization of the Forelli-Rudin construction is also achieved by the second author~\cite{Y-Proc}; precisely, a main object is in exploring the following domain
$$ \Omega_{P, \psi} = \left\lbrace (z,\zeta)\in \Omega \times \mathbb C^d:  P(|\zeta_1|^2, \ldots, |\zeta_d|^2)< \psi(z)  \right\rbrace,$$
with the quasi-homogeneity
$$ \lambda P(x_1, \ldots, x_d)=
 P(\lambda^{\alpha_1}x_1, \ldots, \lambda^{\alpha_d}x_d ),$$
where $\lambda \geq 0$ and $\alpha_1, \ldots, \alpha_d$ are some positive real numbers.
\end{remark}
The followings are examples of weighted Bergman kernels.
\begin{example}\label{Fock}
Consider the Fock-Bargmann space defined by
\[
\mathscr F_{\mu}^2:=\left\lbrace  f \in \mathcal O(\mathbb{C}^n): \dfrac{\mu^n}{\pi^n}\int_{\mathbb{C}^n}  |f(z)|^2 e^{-\mu \|z\|^2}dV(z) < \infty  \right\rbrace.
\] 
The reproducing kernel $K_\mu$ of $\mathscr F_\mu^2$ is called the Fock-Bargmann kernel.
It is known that the Fock-Bargmann kernel has an explicit form:
\begin{align*}
K_\mu(z,w)=   e^{\mu \langle z,w \rangle}. 
\end{align*}
For the computation of the Fock-Bargmann kernel, see \cite{Hall} and \cite{Zhu-Book}.
\end{example}
\begin{example}
Let $\mathscr D$ be an irreducible bounded symmetric domain and $N$ its generic norm.
Then the following integral formula is known:
\[
\int_{\mathscr D} N(z,z)^\mu dV(z)= \dfrac{\chi(0)}{\chi(\mu)}\mbox{Vol}(\mathscr D).
\]
Here, $\chi$ is a polynomial which is called the Hua polynomial \cite[Eq.~(4)]{D}.
For simplicity of the notation we put $c_{\mathscr D, \mu}=\frac{\chi(0)}{\chi(\mu)} \mbox{Vol}(\mathscr D)$.
Consider the weighted Bergman space
\[
\mathscr S_\mu^2:= \left\lbrace  f \in \mathcal O(\mathscr D): \int_{\mathscr D}  |f(z)|^2c_{\mathscr D, \mu}^{-1}  N(z,z) dV(z) < \infty  \right\rbrace.
\]
Then it is known that the reproducing kernel $\widetilde{K}_\mu$ of the weighted Bergman space $\mathscr S_\mu^2$ has a form (cf. \cite[Lemme 1]{D}):
$$\widetilde K_{\mu}(z,w)= N(z,w)^{-g-\mu},$$
where $g$ denotes the genus of $\mathscr D$. For example, if $\mathscr D$ is the unit disc $\mathbb D=\{|z|<1\}$, then $N(z,w)=1-z\overline{w}$ and $g=2$. We refer the reader to \cite{Ar} and \cite[Part V]{Book} for the details concerning Jordan triple systems.
\end{example}
We finish this section with a theorem due to Diederich and Ohsawa \cite{D-O} and further remarks below.
\begin{theorem}\label{thm:D-O}
Let $D$ be a bounded domain in $\mathbb C^n$ and $\varphi, \psi$ admissible weights such that $\mathbb C[z] \subset A^2(D,\varphi) \cap A^2(D,\psi)$.
Suppose that the identity
\begin{align}\label{cond:OD}
K_{D, \varphi}(z,w) = K_{D, \psi}(z,w)
\end{align}
holds on $D \times D$. Then one has $\varphi = \psi$ on $D$.
\end{theorem}
\begin{remark}
Let $c$ be a non-zero constant. If we replace the condition \eqref{cond:OD} by $K_{D, \varphi}(z,w) = c K_{D, \psi}(z,w)$,
then we can derive that $\varphi = c \psi$ with obvious modifications of the argument given in \cite{D-O}.
\end{remark}
\begin{remark}
In the proof of the above theorem, the boundedness of $D$ is only used to prove the implication $(i)\Rightarrow(ii)$:
\begin{enumerate}
\item[$(i)$] $\int_D (\varphi(z)- \psi(z)) P(z) \overline{Q(z)} dV(z)=0$ for any $P, Q \in \mathbb C [z]$;
\item[$(ii)$] $\int_D (\varphi(z)- \psi(z)) h(z) dV(z)=0$ for any $h \in C_0^0(D)$.
\end{enumerate}
A remarkable fact is that the previous theorem also holds for
$D=\mathbb C^n$ and admissible weights $\varphi, \psi$ on $\mathbb C^n$:
Since
\[\int_{\mathbb C^n} (\varphi(z)- \psi(z)) P(z) \overline{Q(z)} dV(z)=0, \quad \mbox{for any } P, Q \in \mathbb C [z],
\]
it follows from the complex Stone-Weierstrass theorem that
there exist real polynomials $P_h(z,\overline z)$ and $Q_h(z,\overline z)$ such that
\begin{align*}
&\left| \int_{\mathbb C^n} (\varphi(z)-\psi(z)) h(z)  dV(z)\right|\\
&=\left| \int_{\mathbb C^n} (\varphi(z)-\psi(z)) (h(z)-P_h(z,\overline{z})-iQ_h(z,\overline{z}))  dV(z)\right|\\
&=\left| \int_{K} (\varphi(z)-\psi(z)) (h(z)-P_h(z,\overline{z})-iQ_h(z,\overline{z}))  dV(z)\right|\\
&\leq \left(\sup_{K} \left|h(z)- P_h(z,\overline{z}) -i Q_h(z,\overline{z}) \right|\right) \| \varphi - \psi\|_{L^1(K)}\\
&< \varepsilon \| \varphi - \psi\|_{L^1(K)},
\end{align*}
for all $\varepsilon>0$ and $h \in C_0^\infty (\mathbb C^n)$. Here, $K:=\mbox{supp}(h)$ and $\varphi, \psi\in L^{1}(\mathbb C^n)$. Moreover, since $C_0^\infty (\mathbb C^n)$ is dense in $L^1(\mathbb C^n)$ and $\varepsilon >0$ is arbitrary, we deduce that $\varphi-\psi \in \left(C_0^\infty(\mathbb C^n)\right)^{\perp}=\{0\} $.
This clearly forces $\varphi=\psi$ on $\mathbb C^n$.
\end{remark}
\section{Fock-Bergmann-Hartogs domain}\label{sec:Fock}
Define the Fock-Bargmann-Hartogs domain $D_{n,m}$ by
\[
D_{n,m}=\{ (z,\zeta) \in \mathbb C^n \times \mathbb C^m : \|\zeta\|^2 < e^{-\mu\|z\|^2} \} , \quad \mu>0.
\]
The Fock-Bargmann-Hartogs domain is an unbounded strongly pseudoconvex domain.
In particular, since $\{(z,0)\in \mathbb C^{n}\times\mathbb C^{m}\}\subset D_{n,m}$, it is non-hyperbolic in the sense of Kobayashi. In~~\cite{KTY}, an explicit description of the automorphism group $\mbox{Aut}(D_{n,m})$ is given as follows.
\begin{theorem}
The automorphism group $\mbox{Aut}(D_{n,m})$ is generated by the following mappings:
\begin{align*}
\varphi_U&: (z,\zeta) \mapsto (Uz, \zeta) , \quad U \in U(n);\\
\varphi_{U'}&: (z,\zeta) \mapsto (z, U'\zeta) , \quad U' \in U(m);\\
\varphi_{v}&: (z,\zeta) \mapsto (z-v, e^{\mu \langle z,v \rangle- \frac{\mu}{2} \|v\|^2 } \zeta)  ,\quad v\in \mathbb C^n.
\end{align*}
\end{theorem}
Let $p$ be an admissible weight function on $\mathbb C^n$.
In this section, we consider the Hartogs domain
$$D_{p}=\{(z,\zeta)\in \mathbb C^n \times \mathbb C^m: \|\zeta\|^2 < p(z) \}$$
such that the Bergman space $A^2(D_{p})$ is a reproducing kernel Hilbert space.
The aim of this section is to show that 
the Fock-Bargmann-Hartogs domain $D_{n,m}$ is the unique domain
which admits $\varphi_{v}$ as an automorphism.

We begin with our investigation with the following simple computational lemma.
\begin{lemma}\label{jacobian}
The Jacobian determinant $J(\varphi_{v}, (z,0))$ is given by
\begin{align*}
J(\varphi_{v}, (z,0))=
\det\begin{pmatrix}
I_{n} & O\\
B & k_v(z) I_m
\end{pmatrix}=k_v(z)^m, \quad B \in \mbox{Mat}_{m\times n} (\mathbb C),
\end{align*}
where $k_w(z)=K_\mu(z,w)/\sqrt{K_\mu (w,w)}$ is the normalized Fock-Bargmann kernel.
\end{lemma}
\begin{proof}
By using an explicit form of the Fock-Bargmann kernel given in Example \ref{Fock}, we know that
the normalized kernel $k_w(z)=K_\mu(z,w)/\sqrt{K_\mu (w,w)}$ has a form:
\begin{align*}
k_w(z)= e^{\mu \langle z,w \rangle - \frac{\mu}{2} \|w\|^2}.
\end{align*}
In particular, if $z=w$, then one has
\begin{align*}
k_z(z)&= e^{\mu \|z\|^2 - \frac{\mu}{2} \|z\|^2}\\
&=  K_{\frac{\mu}{2}} (z,z).
\end{align*}
Thus the mapping $\varphi_{v}$ can be rewritten as
\begin{align*}
\varphi_{v}&: (z,\zeta) \mapsto (z-v, k_v(z) \zeta) .
\end{align*}
Moreover, the Jacobian determinant $J(\varphi_{v}, (z,0))$ has a form
\begin{align*}
J(\varphi_{v}, (z,0))=
\det\begin{pmatrix}
I_{n} & O\\
B & k_v(z) I_m
\end{pmatrix}=k_v(z)^m, \quad B \in \mbox{Mat}_{m\times n} (\mathbb C),
\end{align*}
which completes the proof.
\end{proof}
In the following, for simplicity of the notation, we put $\mathscr{H}_m=A^2(\mathbb C^n, p(z)^m)$.
The next lemma is the key of our argument.
\begin{lemma}\label{keylemma}
Suppose that $\mbox{Aut}(D_{p})$ contains the mapping $\varphi_{v}$.
Then we have the following:
\begin{itemize}
\item[(a)] The weighted Bergman kernel $K_{\mathbb C^n,p^m}$ coincides with the Fock-bargmann kernel $K_{m\mu}$ (up to a constant multiple);
\item[(b)] The space $\mathscr H_m$ coincides with the Fock-Bargmann $\mathscr F_{m\mu}^{2}$.
\end{itemize}
\end{lemma}
\begin{proof}
Using Lemma \ref{jacobian} and the transformation formula \eqref{trans} for $F=\varphi_{v}$, we have
\begin{align*}
K_{D_p}((z,0),(z,0))&= |k_v(z)^m|^2 K_{D_p}(\varphi_{v}(z,0) ,\varphi_{v}(z, 0) )\\
&=|k_v(z)^m|^2 K_{D_p}((z-v,0) , (z-v,0) ).
\end{align*}
Adopting \eqref{restFRC}, we obtain
\begin{align*}
K_{\mathbb C^n,p^m} (z,z)= \dfrac{ m! |k_v(z)^m|^2 K_{\mathbb C^n,p^m}(z-v,z-v)}{ \pi^m}.
\end{align*}
If $z=v$, then we have
\begin{align*}
K_{\mathbb C^n,p^m} (z,z)&=\dfrac{m! K_{\mathbb C^n,p^m}(0,0) |k_z(z)^m|^2}{\pi^m}\\
&= \dfrac{m! K_{\mathbb C^n,p^m}(0,0) K_{m\mu} (z,z)}{\pi^m}.
\end{align*}
Since the weighted Bergman kernel $K_{\mathbb C^n,p^m}$ is determined by its value on the diagonal,
we get
\begin{align*}
K_{\mathbb C^n,p^m} (z,w)= \dfrac{m! K_{\mathbb C^n,p^m}(0,0) K_{m\mu} (z,w)}{\pi^m}.
\end{align*}
Since $K_{\mathbb C^n,p^m} \not \equiv 0$, $K_{\mathbb C^n,p^m}(0,0) $ must be a non-zero constant.\par
The proof of the implication $(\mathrm a) \Rightarrow (\mathrm b)$
proceeds along similar lines as in \cite[p.472]{Zhu}.
To avoid the repetition, we omit the details.
\end{proof}

Now we are ready to prove the following theorem.
\begin{theorem}\label{main1}
Up to a constant multiple of $p$, the Fock-Bargmann-Hartogs domain is the unique Hartogs domain $D_{p}$ which attains $\varphi_{v}$ as an automorphism.
\end{theorem}
\begin{proof}
Suppose that $\mbox{Aut}(D_{p})$ contains the mapping $\varphi_{v}$.
Then it follows from Lemma \ref{keylemma} that $\mathscr H_m=\mathscr F_{m\mu}^{2}$ and
$K_{\mathbb C^n,p^m} (z,w)= c K_{m\mu} (z,w)$. Since $\mathbb C[z]\subset\mathscr F_{m\mu}^{2}$, we conclude that $p(z)= \mbox{const.} e^{-\mu\|z\|^2}$ by using Theorem \ref{thm:D-O} and its consecutive remarks.
\end{proof}
\begin{remark}\label{PL_condition}
In the case when $p(z)e^{\mu \|z\|^2}$ satisfies conditions on $\overline{D_{p}}$ for which a Phragm{\'e}n-Lindel{\"o}f type theorem holds, the maximum principle~(cf.~\cite{BMT, PL} for holomorphic or plurisubharmonic functions) gives an alternative proof of Theorem \ref{main1}:
Let $(z,\zeta) \in \partial D_p$, that is, $\|\zeta\|^2 =p(z)$. Indeed, we have
\begin{align}\label{ineq}
p(z)\|k_v(z)\|^2= \|\zeta\|^2\|k_v(z)\|^2 =\|k_v(z) \zeta\|^2 \leq p(z-v),
\end{align}
for any $(z,\zeta)\in \partial D_p$ and $v \in \mathbb C^n$, albeit $\varphi_v (\partial D_p)\neq\partial D_p$ generically.
Letting $v=z$ in \eqref{ineq}, one gets that 
\[p(z) \| k_z(z)  \|^2 \leq p(0), \quad  \mbox{for any } (z,\zeta)\in \partial D_p.\]
Since $(0,0) \in D_p$, by adopting the maximum principle, we obtain
\[
p(z) \|k_z (z)\|^2 = p(z) e^{\mu \|z\|^2} \equiv p(0), \quad \mbox{for any } (z,\zeta)\in D_p,
\] as desired.
\end{remark}
\section{Cartan-Hartogs domain}\label{CH_dom}
Let $\mathscr{D}$ be an irreducible bounded symmetric domain in $\mathbb C^n$ and $N$ its generic norm.
We next consider the following domain
$$ \Omega^{\mathscr{D}}_N = \left\lbrace (z,\zeta)\in \mathscr{D} \times \mathbb C^m: \|\zeta\|^2 < N(z,z)^\mu \right\rbrace, \quad \mu>0,$$
which is called the Cartan-Hartogs domain. This domain was first introduced by G. Roos and W. Yin and it was studied extensively from various aspects (cf. \cite{AP}, \cite{L}, \cite{Roos}, \cite{Y-CR}). Besides, since $N(z,z)=1-|z|^2$ when $\mathscr{D}$ is the unit disc $\mathbb D$, the Cartan-Hartogs domain includes the Thullen domain 
\[
D_\mu =\{(z,\zeta)\in \mathbb D \times \mathbb C:|z|^2 + |\zeta|^{2/\mu} <1\}
\] as a special case. The automorphism group of the Cartan-Hartogs domain is computed in \cite{ABP}.
\begin{theorem}
The automorphism group of the Cartan-Hartogs domain is generated by the mappings $\Phi$ of the forms:
\begin{align*}
\Phi(z,\zeta)= (\Phi_1 (z,\zeta), \Phi_2 (z,\zeta) ), 
\quad \mbox{for } (z,\zeta)\in\Omega^{\mathscr{D}}_N,
\end{align*}
where $\Phi_1 (z,\zeta)= \phi(z)\in \mbox{Aut}(\mathscr{D})$ and
\begin{align*}
\Phi_2(z,\zeta)= U\cdot \dfrac{N(z_0,z_0)^{\frac{\mu}{2}} }{N(z,z_0)^{\mu} } 
\cdot \zeta, \quad z_0=\phi^{-1}(0), \quad U\in U(m).
\end{align*}
\end{theorem}
We will see an analogue of Theorem \ref{main1} for the Cartan-Hartogs domain.
Let $q$ be an admissible weight of $\mathscr D$ and put $\mathscr {H}'_m=A^2(\mathscr D, q^m)$.
Let us consider the following Hartogs domain
\[
 D_q =\{(z,\zeta) \in \mathscr D \times \mathbb C^n: \|\zeta\|^2 < q(z)\}.
\]
Throughout this section, we assume that $\mbox{Aut}(D_q)$ contains the mapping $\Phi$.
The purpose of this section is to show that the Cartan-Hartogs domain is the unique Hartogs domain $D_q$ whose automorphism group contains the mapping $\Phi$.\par
By a simple computation, we see that the Jacobian determinant $J(\Phi, (z,0))$ is given by
\begin{align*}
J(\Phi, (z,0))&=
\det\begin{pmatrix}
J(\phi, z) & O\\
B & \dfrac{N(z_0,z_0)^{\frac{\mu}{2}} }{N(z,z_0)^{\mu} } \cdot U
\end{pmatrix}\\
&=\dfrac{N(z_0,z_0)^{\frac{m\mu}{2}} }{N(z,z_0)^{m\mu} } \det J(\phi, z) ,
\end{align*}
where  $B \in \mbox{Mat}_{m\times n} (\mathbb C)$.
Then it is known that $\det J(\phi,z_0)$ satisfies the following relation (cf. \cite[Lemme 1]{D}):
\begin{align*}
|\det J(\phi,z_0)|^2 = N(z_0,z_0)^{-g}.
\end{align*}
Then, by an argument similar to that in Section \ref{sec:Fock}, we deduce
\begin{align*}
K_{D_q}((z,0),(z,0))&=\left| \dfrac{N(z_0,z_0)^{\frac{m\mu}{2}}}{N(z,z_0)^{m\mu}}\right|^2|\det J(\phi,z)|^2  K_{D_q} ((\phi(z),0),(\phi(z),0)).
\end{align*}
Putting $z=z_0$, then we obtain
\begin{align}\label{relation}
K_{D, q^m}(z_0,z_0)&= K_{D, q^m}(0,0) N(z_0,z_0)^{-m\mu-g}\\
&=K_{D, q^m}(0,0) \widetilde{K}_{m\mu} (z_0,z_0) . \notag
\end{align}
Since every irreducible bounded symmetric domain $\mathscr D$ is homogeneous, there exists $f \in \mbox{Aut}(\mathscr D)$ with $f(z)=0$ for any $z \in \mathscr D$.
Thus the equation \eqref{relation} holds for any $z \in \mathscr D$.
Hence, by using the same argument as that in Section \ref{sec:Fock}, we obtain the following lemma:
\begin{lemma}
Suppose that $\mbox{Aut}(D_{q})$ contains the mapping $\Phi$.
Then we have the following:
\begin{itemize}
\item[(a)] The weighted Bergman kernel $K_{\mathbb C^n,q^m}$ coincides with  $\widetilde{K}_{m\mu}$ (up to a constant multiple);
\item[(b)] The space $\mathscr{H}'_m$ coincides with the $\mathscr S_{m\mu}^{2}$.
\end{itemize}
\end{lemma}
Combining this lemma with Theorem \ref{thm:D-O}, we obtain the following theorem.
\begin{theorem}\label{main2}
The Cartan-Hartogs domain is the unique Hartogs domain $D_q$ 
which possesses $\Phi$ as an automorphism.
\end{theorem}
The proof proceeds along the same lines as the proof of Theorem \ref{main1}, replacing $\mathbb C[z]\subset\mathscr F_{m\mu}^2$ by $\mathbb C[z]\subset\mathscr S_{m\mu}^2$.
To avoid repetition we omit the proof of this theorem.\par
As a special case of Theorem \ref{main2}, we obtain the following (cf. \cite[Theorem 2.3.5]{J-P}):
Let $r$ be an admissible weight of $\mathbb D$. Define $D_{r}=\{(z,\zeta)\in \mathbb D \times \mathbb C: |\zeta|^2 < r(z)\}$ and $\varphi_a: D_{r}\rightarrow \mathbb C^2$ by
$$\varphi_a: D_{r}\rightarrow \mathbb C^2, (z,\zeta) \mapsto  \left( \dfrac{z-a}{1-\overline{a}z},\dfrac{(1-|a|^2)^{\frac{\mu}{2}}}{(1-\overline{a}z)^{\mu}}\zeta\right),$$
where $a\in\mathbb D$. As is well-known, $\varphi_a$ is an automorphism  of the Thullen domain.
\begin{corollary}
The Thullen domain is the unique Hartogs domain $D_{r}$ 
which attains $\varphi_a$ as an automorphism.
\end{corollary}
\section{Concluding remark}\label{C_Rmk}
We conclude this paper with a remark on our argument.\par
Let $\Omega$ be a bounded homogeneous domain and $P$ an admissible weight of $\Omega$.
Consider the Hartogs domain $\Omega_P=\{(z,\zeta)\in\Omega\times\mathbb C^m:\|\zeta\|^2 < P(z)\}$ with the condition $\mathbb C[z]\subset A^2(\Omega, P^m)$.
Suppose that the automorphism group of $\Omega_P$ contains 
\begin{align}\label{A}
A=\left\lbrace \varphi(z,\zeta)=(\varphi_1(z), \varphi_2(z,\zeta)): 
\begin{array}{ll}
\varphi_1 \in \mbox{Aut}(\Omega), \varphi_2(z,0)=0,\\
|J(\varphi,(z,0))|^2=\mbox{const.}K_{\Omega, P^m}(z,z) 
\end{array}
\right\rbrace.
\end{align}
Then the same argument as in the previous section tells us that
$\Omega_P$ is the unique Hartogs domain $\Omega_Q=\{(z,\zeta)\in \Omega \times \mathbb C^m: \|\zeta\|^2 < Q(z)\}$ such that $A\subset \mbox{Aut}(\Omega_Q)$.

Indeed, if $A\subset \mbox{Aut}(\Omega_Q)$, then we have
\begin{align*}
K_{\Omega, Q^m}(z,z)&=|J(\varphi, (z,0))|^2 K_{\Omega, Q^m}(\varphi_1(z),\varphi_1(z))\\
&= \mbox{const.} K_{\Omega, P^m}(z,z) K_{\Omega, Q^m}(\varphi_1(z),\varphi_1(z)).
\end{align*}
Since $\Omega$ is homogeneous, there exists $\widetilde{\varphi}_1 \in \mbox{Aut}(\Omega)$  such that 
$\widetilde{\varphi}_1(z_0)=0$ for any $z_0 \in \Omega$.
This implies that
\[
K_{\Omega, Q^m}(z,z)=
\mbox{const.} K_{\Omega, P^m}(z,z) K_{\Omega, Q^m}(0,0).
\]
Moreover, we can show that
\begin{align*}
&K_{\Omega,P^m}(z,w)=\mbox{const.}K_{\Omega,Q^m}(z,w);\\
& A^2(\Omega, P^m) = A^2(\Omega, Q^m).
\end{align*}
Using Theorem \ref{thm:D-O} in conjunction with the condition $\mathbb C[z] \subset A^2(\Omega, P^m)$,
we thus prove that $P= \mbox{const.} Q$.\par
As we have seen in this paper, the Cartan-Hartogs domain is an example whose automorphism group contains a subset of the form \eqref{A}.
We do not know such examples if the base domain of a Hartogs domain is a non-symmetric homogeneous domain.
It would be interesting to study which kind of condition ensures the existence of a subset of the form \eqref{A}.

\acknowledgement{The authors would like to express their gratitude to Professors Hong Rae Cho, Kang-Tae Kim and Takeo Ohsawa for sincere and helpful discussion. The authors also extend their thanks to Jae-Cheon Joo and Van Thu Ninh for their valuable comments on this paper.}

\end{document}